\documentclass[12pt]{amsart}
\usepackage{amscd}
\def\frk{\frak}               

\def\Phi{{\frk n}}
\def\Phi{{\frk N}}
\def\opn#1#2{\def#1{\operatorname{#2}}} 
\opn\chara{char} \opn\length{\ell} \opn\pd{pd} \opn\rk{rk}
\opn\projdim{proj\,dim} \opn\injdim{inj\,dim} \opn\rank{rank}
\opn\depth{depth} \opn\sdepth{sdepth} \opn\fdepth{fdepth}
\opn\grade{grade} \opn\height{height} \opn\embdim{emb\,dim}
\opn\codim{codim}  \opn\min{min} \opn\max{max}

\opn\Tr{Tr} \opn\bigrank{big\,rank}
\opn\superheight{superheight}\opn\lcm{lcm}
\opn\trdeg{tr\,deg}
\opn\reg{reg} \opn\lreg{lreg} \opn\ini{in} \opn\lpd{lpd}
\opn\size{size}
\opn\div{div} \opn\Div{Div} \opn\cl{cl} \opn\Cl{Cl}
\opn\Spec{Spec} \opn\Supp{Supp} \opn\supp{supp} \opn\Sing{Sing}
\opn\Ass{Ass} \opn\Min{Min}
\opn\Ann{Ann} \opn\Rad{Rad} \opn\Soc{Soc}
\opn\Im{Im} \opn\Ker{Ker} \opn\Coker{Coker} \opn\Am{Am}
\opn\Hom{Hom} \opn\Tor{Tor} \opn\Ext{Ext} \opn\End{End}
\opn\Aut{Aut} \opn\id{id}  \opn\deg{deg}

\opn\nat{nat}
\opn\pff{pf}
\opn\Pf{Pf} \opn\GL{GL} \opn\SL{SL} \opn\mod{mod} \opn\ord{ord}
\opn\Gin{Gin} \opn\Hilb{Hilb}
\opn\aff{aff} \opn\con{conv} \opn\relint{relint} \opn\st{st}
\opn\lk{lk} \opn\cn{cn} \opn\core{core} \opn\vol{vol}
\opn\link{link} \opn\star{star}
\opn\gr{gr}

\def\pot#1#2{#1[\kern-0.28ex[#2]\kern-0.28ex]}

\opn\dirlim{\underrightarrow{\lim}}
\opn\inivlim{\underleftarrow{\lim}}
\let\to=\rightarrow

\def\Implies{\ifmmode\Longrightarrow \else
        \unskip${}\Longrightarrow{}$\ignorespaces\fi}
\def\implies{\ifmmode\Rightarrow \else
        \unskip${}\Rightarrow{}$\ignorespaces\fi}
\def\iff{\ifmmode\Longleftrightarrow \else
        \unskip${}\Longleftrightarrow{}$\ignorespaces\fi}

\let\:=\colon
\newtheorem{Theorem}{Theorem}[section]
\newtheorem{Lemma}[Theorem]{Lemma}
\newtheorem{Corollary}[Theorem]{Corollary}
\newtheorem{Proposition}[Theorem]{Proposition}
\newtheorem{Remark}[Theorem]{Remark}

\newtheorem{Example}[Theorem]{Example}

\let\epsilon\varepsilon
\let\phi=\varphi
\let\kappa=\varkappa
\def\qed{\ifhmode\textqed\fi
      \ifmmode\ifinner\quad\qedsymbol\else\dispqed\fi\fi}
\def\textqed{\unskip\nobreak\penalty50
       \hskip2em\hbox{}\nobreak\hfil\qedsymbol
       \parfillskip=0pt \finalhyphendemerits=0}
\def\dispqed{\rlap{\qquad\qedsymbol}}

\opn\dis{dis}
\def\pnt{{\raise0.5mm\hbox{\large\bf.}}}

\opn\Lex{Lex}

\begin{document}

\title{\bf The Stanley conjecture on monomial almost complete   intersection ideals}

\author{Mircea Cimpoea\c s }

\thanks{The  support from  the UEFISCDI grant 247/2011 of Romanian Ministry of Education, Research and Innovation is gratefully acknowledged.}
\address{Mircea Cimpoea\c s, Simion Stoilow Institute of Mathematics, Research unit 5, P.O.Box 1-764, Bucharest 014700, Romania}
\email{mircea.cimpoeas@imar.ro}

\maketitle
\begin{abstract} Let $I$ be a monomial almost complete intersection ideal of a polynomial algebra $S$ over a field. Then Stanley's Conjecture holds for $S/I$ and $I$.

  \vskip 0.4 true cm
 \noindent
  {\it Key words } : Monomial Ideals,  Stanley decompositions, Stanley depth.\\
 {\it 2000 Mathematics Subject Classification: Primary 13C15, Secondary 13F20, 13F55, 13P10.}
\end{abstract}

\section*{Introduction}

Let $K$ be a field and $S=K[x_1,\ldots,x_n]$ the polynomial ring over $K$.
Let $M$ be a $\mathbb Z^n$-graded $S$-module. A {  \em Stanley decomposition} of $M$ is a direct sum $\mathcal D: M = \bigoplus_{i=1}^rm_i K[Z_i]$ as $K$-vector space, where $m_i\in M$, $Z_i\subset\{x_1,\ldots,x_n\}$ such that $m_i K[Z_i]$ is a free $K[Z_i]$-module. We define $\sdepth(\mathcal D)=\min_{i=1}^r |Z_i|$ and $\sdepth_S(M)=\max\{\sdepth(\mathcal D)|\;\mathcal D$ is a Stanley decomposition of $M\}$. The number $\sdepth_S(M)$ is called the {\em Stanley depth} of $M$. It is conjectured by Stanley \cite{stan} that $\depth_S(M) \leq \sdepth_S(M)$ for all $\mathbb Z^n$-graded $S$-modules $M$. Herzog, Vladoiu and Zheng show in \cite{hvz} that this invariant can be computed in a finite number of steps if $M=I/J$, where $J\subset I\subset S$ are monomial ideals.
In this paper, we prove that if $I\subset S$ is a monomial ideal generated by $m$ monomials, then, there exists a variable $x_j$ which appears in at least $\left\lceil \frac{m}{k} \right\rceil$ generators, where
$k=\max\{|P|:P\in Ass(S/I)\}$, see Lemma \ref{m}. Using this lemma, we prove that Stanley's Conjecture holds for  $S/I$ and $I$, when $I$ has a small number of generators, with respect to $\depth(S/I)$ and $k$, see Theorem \ref{ma}, in particular this is the case when $I$ is a monomial almost complete intersection ideal (see Corollary \ref{main}).

\section{Stanley depth}

Firstly, we recall several results.

\begin{Proposition}\cite[Proposition 1.2]{mir}\label{1}
Let $I \subset S$ be a monomial ideal, minimally generated by $m$ monomials. Then $\sdepth(S/I) \geq n-m$.
\end{Proposition}
\begin{Theorem}\cite[Theorem 2.3]{okazaki}\label{2}
Let $I \subset S$ be a monomial ideal, minimally generated by $m$ monomials. Then $\sdepth(I) \geq n-\left\lfloor m/2 \right\rfloor$.
\end{Theorem}

\begin{Proposition}\cite[Theorem 1.4]{mir}\label{3}
Let $I\subset S$ be a monomial ideal such that $I=v(I:v)$, for a monomial $v\in S$. Then $\sdepth(S/I)=\sdepth(S/(I:v))$, $\sdepth(I)=\sdepth(I:v)$.
\end{Proposition}

If $v\in S$ is a monomial, we define the \emph{support of $v$}, to be $\supp(v):=\{x_j:\;x_j|v\}$. Also, we denote $\deg_{x_j}(v):=\max\{t:\;x_j^t|v\}$. Let $I=(v_1,\ldots,v_m)\subset S$, $I\not =S$ be a monomial ideal, where  $G(I)=\{v_1,\ldots,v_m\}$ is a minimal system of monomial generators of $I$. We denote $t_j:=|\{i:\;x_j|v_i\}|$ and $\mathcal V:=\bigcup_{i=1}^m \supp(v_i)$.

\begin{Remark} \label{r}{\em
It is well known that $\depth(S/I)\leq \min\{\dim(S/P):\;P\in \Ass(S/I)\}=\min\{n-|P|:\;P\in \Ass(S/I)\}$ by \cite[Proposition 1.2.13]{bh}. Denote $k=\max\{|P|:\;P\in \Ass(S/I) \}$. In particular, we get $k\leq n-\depth(S/I)$. We have $k\leq m$ because a prime ideal $P\in \Ass(S/I)$ has the form $I:u$ for some monomial $u\not \in I$.}
\end{Remark}
 With these notations, we have the following lemma:

\begin{Lemma}\label{m}
There exists a $j\in [n]:=\{1,\ldots,n\}$ such that $t_j\geq \left\lceil m/k \right\rceil$.
\end{Lemma}

\begin{proof}
We use induction on $k\geq 1$ and $\epsilon(I)=\sum_{i=1}^m \deg(v_i)$. If $k=1$, it follows that $I$ is principal, and therefore, we can assume that $I=(v_1)$ and $m=1$.  If we chose $x_j\in \supp(v_1)$, it follows that $t_j=1=\left\lceil m/1 \right\rceil$ and thus we are done. If $\epsilon(I)=k$, it follows $\epsilon(I)=k\leq m\leq \epsilon(I)$ by Remark \ref{r}. Thus $I$ is
generated by $m=k$ variables, and there is nothing to prove. Assume $k\geq 2$ and $\epsilon(I)>k$.

Assume that $(\mathcal V)\subset \sqrt{I}$. Since, for any monomial $v\in G(\sqrt{I})$ we have $\supp(v)\subset \mathcal V$ it follows that the prime ideal $P:=(\mathcal V)$  contains also $\sqrt{I}$. Thus $P=
\sqrt{I}$ is a prime ideal and $P=(\mathcal V)$. Therefore, $I$ is $P$-primary. Since $k=|P|$, by reordering the variables, we may assume that $P=(x_1,\ldots,x_k)$. We  may also assume that $v_1=x_1^{a_1}$, \ldots, $v_r=x_k^{a_k}$ for some positive integers $a_l$, where $l\in [k]$.
Since $\mathcal V=\{x_1,\ldots,x_k\}$, it follows that $t_j=0$ for all $j>k$.
Note that $\sum_{i=1}^m |\supp(v_i)| = \sum_{j=1}^k t_j$. Indeed, each variable $x_j$ appear in the supports of exactly $t_j$ monomials from the set $\{v_1,\ldots,v_m\}$. Now, we claim that there exists a $t_j\geq \left\lceil m/k \right\rceil$. Indeed, if this is not the case, then we get $m\leq \sum_{i=1}^m |\supp(v_i)| = \sum_{j=1}^k t_j < \sum_{j=1}^k (m/k) = m$, a contradiction. 

If there exists a variable, let us say $x_n$, such that $x_n \in \mathcal V$ and $x_n\notin \sqrt{I}$, we consider the ideal $I'=(I:x_n)$. Obviously, $I'=(v'_1,\ldots,v'_m)$, where $v'_i=v_i/x_n$ if $x_n|v_i$ and $v'_i=v_i$ otherwise. For all $j\in[n]$, we denote $t_j'=|\{i:\;x_j|v'_i\}|$. Note that $t_j=t'_j$ for all $j\in[n-1]$, and $t_n\geq t'_n$. If we denote $\mathcal V'=\bigcup_{i=1}^m\supp(v'_i)$, we have $\mathcal V'\subset \mathcal V$. Note that $\Ass(S/I')\subset \Ass(S/I)$ because of the injection $S/I'\to S/I$ induced by the multiplication with $x_n$.
It follows that $k'=\max\{|P'|:\;P'\in\Ass(S/I')\} \leq k$. Since $\epsilon(I')=\sum_{i=1}^m \deg(v'_i)<\epsilon(I)$, by induction hypothesis, there exists a $j\in [n]$, such that  $t_j\geq t'_j\geq \left\lceil m/k' \right\rceil \geq \left\lceil m/k \right\rceil$.
\end{proof}

\begin{Example}{\em
Let $I=(x_1^3,x_1x_2, x_2x_3, x_3x_4, x_4^2)\subset S:=K[x_1,x_2,x_3,x_4]$. Then $I=(x_1^3,x_2,x_4)\cap (x_1^3, x_2,x_3,x_4^2)\cap (x_1,x_3, x_4^2)$ is the primary decomposition of $I$. Therefore $\Ass(S/I)=\{(x_1,x_2,x_4),(x_1,x_3,x_4), (x_1,x_2,x_3,x_4)\}$ and $k=\max\{|P|:\;P\in\Ass(S/I)\}=4$. The (minimal) number of monomial generators of $I$ is $m=5$. We have $\left\lceil m/k \right\rceil = 2$ and, indeed, $x_1$, for example, appears in two generators of $I$. This example also shows that the bound $\left\lceil m/k \right\rceil$ is, in general, the best possible.}
\end{Example}

\begin{Lemma}\label{e}
Let $s\geq k\geq 2$ be two integers and let $m$ be a positive integer. Then

 \begin{enumerate}
 \item{} $m-\left\lceil \frac{m}{k} \right\rceil \leq s-1$ if and only if $m\leq s-1+\left\lceil \frac{s}{k-1} \right\rceil$,
 \item{} $ \left\lfloor \frac{m-\left\lceil \frac{m}{k} \right\rceil }{2} \right\rfloor \leq s-2$ if and only if $m\leq 2s-3 + \left\lceil \frac{2s-2}{k-1} \right\rceil$.
     \end{enumerate}
\end{Lemma}

\begin{proof}
 Note that $m-\left\lceil \frac{m}{k} \right\rceil \leq s-1$ if and only if $m - \frac{m}{k}<s$. This is equivalent with $m< \frac{sk}{k-1} = s + \frac{s}{k-1}$. Similarly, we get the second equivalence.
\end{proof}

\begin{Theorem}\label{ma}
Let $I\subset S=K[x_1,\ldots,x_n]$ be a monomial ideal, minimally generated by $m$ monomials,
 $k=\max\{|P|:\;P\in \Ass(S/I)\}$, and  $s\geq k$ be  an integer. Then
 \begin{enumerate}
\item{} If $m\leq s-1+\left\lceil \frac{s}{k-1} \right\rceil$, then $\sdepth(S/I)\geq n-s$.
\item{} If $m\leq 2s-3 + \left\lceil \frac{2s-2}{k-1} \right\rceil$, then $\sdepth(I)\geq n-s+1$.
\end{enumerate}
If $\depth(S/I)=n-s$ then $(1)$ and $(2)$ imply the Stanley Conjecture for $S/I$, respectively for $I$.
\end{Theorem}

\begin{proof}
If $I$ is principal, then $k=m=1$ and $\sdepth(S/I)\geq n-1$ by Proposition \ref{1}. Assume $m\geq 2$, $G(I)=\{v_1,\ldots,v_m\}$ and set $\epsilon(I):=\sum_{i=1}^m \deg(v_i)$. We use induction on $\epsilon(I)$.
If $\epsilon(I)=m$ it follows that $I$ is generated by $m$ variables. Therefore $k=|I|=m$ and so $\sdepth_S(S/I)=n-m=n-k\geq n-s$.
Also by \cite[Theorem 2.2]{par} and \cite[Lemma 3.6]{hvz}, $\sdepth(I)=n - \left\lfloor m/2 \right\rfloor \geq n-m+1 = n-k+1 \geq n-s+1$.

Assume $\epsilon(I)>m$. According to Lemma \ref{m}, we can assume that $r:=t_n\geq \left\lceil m/k \right\rceil$ after renumbering of variables.
If $r=m$, then $x_n|v_j$ for all all $i\in [m]$ and thus $I=x_n(I:x_n)$. According to Proposition \ref{3}, $\sdepth(S/I)=\sdepth(S/(I:x_n))$ and $\sdepth(I)=\sdepth(I:x_n)$.
As in the proof of Lemma \ref{m}, we have $k'=\max\{|P'|:\;P'\in Ass(S/(I:x_n))\}\leq k$ and so our statement holds for $S/(I:x_n)$ by induction hypothesis. Thus $\sdepth(S/I) = \sdepth(S/(I:x_n))\geq n-s$ and $\sdepth(I) = \sdepth(I:x_n) \geq n-s+1$.

We consider now the case $r<m$. By reordering the generators of $I$, we may assume that $x_n|v_1$,\ldots,$x_n|v_r$ and $v_n\not | v_{r+1}$,\ldots, $x_n\not | v_m$. Let $S'=K[x_1,\ldots,x_{n-1}]$. We write:
\[ (*)\;\;\; S/I = (S'/(I \cap S'))\oplus x_n(S/(I:x_n)),\;\; \mbox{and}\; \; I = (I\cap S')\oplus x_n(I:x_n) \]
the direct sum being of linear spaces.
By Proposition  \ref{1}, Theorem \ref{2} and Lemma \ref{e}, it follows that:
\[ \sdepth_{S'}(S'/(I\cap S')) \geq (n-1)-(m-r) \geq n-(m-\left\lceil \frac{m}{k} \right\rceil +1)\geq n-s\; \mbox{and}\]
\[ \sdepth_{S'}(I\cap S') \geq (n-1)-\left\lfloor \frac{m-r}{2} \right\rfloor \geq n- \left\lfloor \frac{m-\left\lceil \frac{m}{k} \right\rceil}{2} \right\rfloor  +1 \geq n-s+1,\]
because $r\geq \left\lceil m/k \right\rceil$.
Let $m'$ be the minimal number of generators of $I:x_n$. In the first case, we have $m'\leq m\leq  s-1+\left\lceil \frac{s}{k-1} \right\rceil\leq
 s-1+\left\lceil \frac{s}{k'-1} \right\rceil$ because $k'\leq k$.
By induction hypothesis, we get $\sdepth(S/(I:x_n))\geq n-s$. Similarly, $\sdepth(I:x_n)\geq n-s+1$ in the second case.
Using the decompositions $(*)$, we  obtain  Stanley decompositions of $S/I$, $I$ with the Stanley depth $\geq n-s$, respectively $\geq n-s+1$.
\end{proof}

\begin{Corollary}\label{main} Let $I$ be a monomial almost complete intersection ideal. Then Stanley's Conjecture holds for  $S/I$ and $I$.
 \end{Corollary}
 \begin{proof} Let  $s:=n-\depth(S/I)$. Then $m=s+1$. Since $s\geq k$ by Remark \ref{r} we have  $m\leq s-1+\left\lceil \frac{s}{k-1} \right\rceil$  and we see  that $S/I$ satisfies Stanley's Conjecture by (1) of the above theorem. If $s\geq 2$ then similarly  $m\leq 2s-3 + \left\lceil \frac{2s-2}{k-1} \right\rceil$ and so $I$ satisfies Stanley's Conjecture by (2) of the above theorem. But if $s=1$ then $k=1$ and it follows that $I$ is principal, in which case clearly Stanley's Conjecture holds.
\end{proof}

\begin{Remark}{\em
Note that the decomposition used in the proof of Theorem $1.8$ is also useful to check  Stanley's Conjecture for monomial ideals (and their quotient rings), which do not satisfy the hypothesis of the quoted theorem.
For example, consider the ideal $I=(x_1x_2,x_1x_3,x_1x_4,x_2x_3,x_2x_4,x_3x_4)\subset S:=K[x_1,\ldots,x_4]$. We have $(I:x_1)=(x_2,x_3,x_4)$. If we denote $S':=K[x_1,x_2,x_3]$, then $I':=I\cap S' = (x_1x_2,x_1x_3,x_2x_3)$. One can easily check that $\sdepth_{S'}(S'/I')=1$ and $\sdepth_{S'}(I')=2$. Using the decomposition $I=I' \oplus x_1(I:x_1)$, it follows, as in the proof of Theorem $1.8$, that $\sdepth(I)\geq 2$ and $\sdepth(S/I)\geq 1$. On the other hand, it is well known that $\depth(S/I)=1$. Of course, $I$ is a squarefree Veronese ideal, and we already know that $I$ and $S/I$ satisfy the Stanley conjecture, by \cite[Corollary 1.2]{mirc}.}
\end{Remark}

\end{document}